\documentclass[letterpaper, 10 pt, conference]{ieeeconf}  % Comment this line out if you need a4paper
\usepackage{amsmath}
\usepackage{graphicx}
\newtheorem{lemma}{Lemma}
\newtheorem{remark}{Remark}
\newtheorem{theorem}{Theorem}

\IEEEoverridecommandlockouts                              % This command is only needed if 
\title{\LARGE \bf
Minimum time control of heterodirectional linear coupled hyperbolic PDEs
}

\author{Jean Auriol$^{1}$ and Florent Di Meglio$^{2}$% <-this % stops a space
\thanks{$^{1}$Jean Auriol is with MINES ParisTech, PSL Research University, CAS - Centre automatique et syst\`emes, 60 bd St Michel 75006 Paris, France. 
        {\tt\small jean.auriol@mines-paristech.fr}}%
\thanks{$^{2}$Florent Di Meglio is with MINES ParisTech, PSL Research University, CAS - Centre automatique et syst\`emes, 60 bd St Michel 75006 Paris, France. 
        {\tt\small florent.di\_meglio@mines-paristech.fr}}%
}

\begin{document}

\maketitle
\thispagestyle{empty}
\pagestyle{empty}

%%%%%%%%%%%%%%%%%%%%%%%%%%%%%%%%%%%%%%%%%%%%%%%%%%%%%%%%%%%%%%%%%%%%%%%%%%%%%%%%
\begin{abstract}

We solve the problem of stabilizing a general class of linear first-order hyperbolic systems. Considered systems feature an arbitrary number of coupled transport PDEs convecting in either direction. Using the backstepping approach, we derive a full-state feedback law and a boundary observer enabling stabilization by output feedback. Unlike previous results, finite-time convergence to zero is achieved in the theoretical lower bound for control time. 

\end{abstract}

%%%%%%%%%%%%%%%%%%%%%%%%%%%%%%%%%%%%%%%%%%%%%%%%%%%%%%%%%%%%%%%%%%%%%%%%%%%%%%%%
\section{INTRODUCTION}
 
This article solves the problem of boundary stabilization of a general class of coupled heterodirectional linear first-order hyperbolic systems of Partial Differential Equations (PDEs) in minimum time, with arbitrary numbers $m$ and $n$ of PDEs in each direction and with actuation applied on only one boundary. 

First-order hyperbolic PDEs are predominant in modeling of traffic flow \cite{amin2008stability}, heat exchanger \cite{xu2002exponential}, open channel flow~\cite{coron1999lyapunov},~\cite{de2003boundary} or multiphase flow~\cite{di2011dynamics, djordjevic2010boundary, dudret2012stability}.
Research on controllability and stability of hyperbolic systems have first focused on explicit computation of the solution along the characteristic curves in the framework of the $C^1$ norm~\cite{greenberg1984effect},~\cite{li1994global},~\cite{qin1985global}. Later, Control Lyapunov Functions methods emerged, enabling the design of dissipative boundary conditions for nonlinear hyperbolic systems \cite{coron2009control, coron2008dissipative}. In~\cite{coron2013local} control laws for a system of two coupled nonlinear PDEs are derived, whereas in \cite{buenaventura2012dynamic, coron2008dissipative, prieur2012iss, prieur2008robust, santos2008boundary} sufficient conditions for exponential stability are given for various classes of quasilinear first-order hyperbolic system. These conditions typically impose restrictions on the magnitude of the coupling coefficients.

In \cite{vazquez2011backstepping} a backstepping transformation is used to design a single boundary output-feedback controller. This control law yields $H^2$ exponential stability of closed loop 2-state heterodirectional linear and quasilinear hyperbolic system for arbitrary large coupling coefficients. A similar approach is used in \cite{di2013stabilization} to design output feedback laws for a system of coupled first-order hyperbolic linear PDEs with $m=1$ controlled negative velocity and $n$ positive ones. The generalization of this result to an arbitrary number $m$ of controlled negative velocities is presented in \cite{hu2015control}. There, the proposed control law yields finite-time convergence to zero, but the convergence time is larger than the minimum control time, derived in \cite{li2010strong, woittennek2009flatness}. This is due to the presence of non-local coupling terms in the targeted closed-loop behavior.

The main contribution of this paper is a minimum time stabilizing controller. More precisely, a proposed boundary feedback law ensures finite-time convergence of all states to zero in minimum-time. This minimum-time, defined in~\cite{li2010strong},~\cite{woittennek2009flatness} is the sum of the two largest time of transport in each direction.  

Our approach is the following. Using a backstepping approach (with a Volterra transformation) the system is mapped to a \textit{target} system with desirable stability properties. This target system is a copy of the original dynamics with a modified in-domain coupling structure. More precisely, the target system is designed as an exponentially stable cascade.  A full-state feedback law guaranteeing exponential stability of the zero equilibrium in the $\mathcal{L}^2$-norm is then designed. This full-state feedback law requires full distributed measurements.  For this reason we derive a boundary observer relying on measurements of the states at a single boundary (the anti-collocated one). Similarly to the control design, the observer error dynamics are mapped to a target system using a Volterra transformation. Along with the full-state feedback law, this yields an output feedback controller amenable to implementation. 

The main technical difficulty of this paper is to prove well-posedness of the Volterra transformation. Interestingly, the transformation kernels satisfy a system of equations with a cascade structure akin to the target system one. This structure enables a recursive proof of existence of the transformation kernels.

The paper is organized as follows. In Section II we introduce the model equations and the notations. In Section III we present the stabilization result: the target system and its properties are presented in Section III-A. In Section III-B we derive the backstepping transformation. Section IV contains the main technical difficulty of this paper which is the proof of well-posedness of the kernel equations. In Section IV-A we transform the kernel equations into an integral equation using the method of characteristics. In Section IV-B we solve the integral equations using the method of successive approximations. In Section V we present the control feedback law and its properties. In Section VI we present the uncollocated observer design. In Section VII we give some simulation results. Finally in Section VIII we give some concluding remarks

\section{Problem Description}
\subsection{System under consideration}

We consider the following general linear hyperbolic system 

\begin{align}
u_t(t,x)+\Lambda^+u_x(t,x)&=\Sigma^{++}u(t,x)+\Sigma^{+-}v(t,x) \label{u_pde}\\
v_t(t,x)-\Lambda^-v_x(t,x)&=\Sigma^{-+}u(t,x)+\Sigma^{--}v(t,x) \label{v_pde}
\end{align}
with the following linear boundary conditions 

\begin{align}
u(t,0)=Q_0v(t,0), \quad v(t,1)=R_1u(t,1)+U(t) \label{IC}
\end{align}
where

\begin{align}
u=(u_1 \dots u_n)^T, \quad v=(v_1 \dots v_m)^T 
\end{align}

\begin{align}
\Lambda^+=\begin{pmatrix}\lambda_1 & &0\\ &\ddots& \\ 0 & &\lambda_n \end{pmatrix}, \quad \Lambda^-=\begin{pmatrix}\mu_1 & &0\\ &\ddots& \\ 0 & &\mu_m \end{pmatrix}
\end{align}
with constant speeds :

\begin{align}
-\mu_m< \dots < -\mu_1< 0 <\lambda_1 \leq \dots \leq \lambda_n
\end{align}
and constant coupling matrices as well as the feedback control input 

\begin{align}
\Sigma^{++}&=\{\sigma^{++}_{ij}\}_{1\leq i \leq n, 1\leq j \leq n} \quad \Sigma^{+-}=\{\sigma^{+-}_{ij}\}_{1\leq i \leq n, 1\leq j \leq m} \\
\Sigma^{-+}&=\{\sigma^{-+}_{ij}\}_{1\leq i \leq m, 1\leq j \leq n} \quad \Sigma^{--}=\{\sigma^{--}_{ij}\}_{1\leq i \leq m, 1\leq j \leq m} \\
Q_0&=\{q_{ij}\}_{1\leq i \leq n, 1\leq j \leq m} \quad R_1=\{\rho_{ij}\}_{1\leq i \leq m, 1\leq j \leq n}
\end{align}

\subsection{Control problem}

The goal is to design feedback control inputs $U(t)=(U_1(t), \dots, U_m(t))^T$ such that the zero equilibrium is reached in minimum time $t=t_F$, where 
\begin{align}
t_F=\dfrac{1}{\mu_1}+\dfrac{1}{\lambda_1} \label{t_F}
\end{align}

\section{Control design}
The control design is based on the backstepping approach: using a Volterra transformation, we map the system \eqref{u_pde}-\eqref{IC} to a target system with desirable properties of stability.

\subsection{Target system}
\subsubsection{Target system design} 
We map the system \eqref{u_pde}-\eqref{IC} to the following system

\begin{multline}
\alpha_t(t,x)+\Lambda^+\alpha_x(t,x) = \Sigma^{++}\alpha(t,x)+\Sigma^{+-}\beta(t,x) \\
+\int^x_0 C^+(x,\xi)\alpha(t,\xi) d\xi + \int^x_0 C^-(x,\xi)\beta(t,\xi) d\xi
\label{a_pde}
\end{multline}

\begin{align}
\beta_t(t,x)-\Lambda^-\beta_x(t,x)=\Omega(x)\beta(t,x) 
\label{b_pde}
\end{align}
with the following boundary conditions

\begin{align}
\alpha(t,0)=Q_0\beta(t,0) \quad \beta(t,1)=0
\label{IC_ab}
\end{align}
where $C^+$ and $C^-$ are $L^{\infty}$ matrix functions on the domain 

\begin{align}
\mathcal{T}=\{0 \leq \xi \leq x \leq 1 \}
\end{align}
while $\Omega \in L^{\infty}(0,1)$ is an upper triangular matrix with the following structure
\begin{align}
\Omega(x)= \begin{pmatrix} \omega_{1,1}(x) & \omega_{1,2}(x) & \dots & \omega_{1,m}(x) \\ 0 & \ddots & \ddots & \vdots \\  \vdots & \ddots  & \omega_{m-1,m-1}(x) &\omega_{m-1,m}(x)  \\ 0 &\dots & 0 & \omega_{m,m}(x) \end{pmatrix}
\end{align}
\\
This system is designed as a copy of the original dynamics, from which the coupling terms of \eqref{v_pde} are removed. The integral coupling appearing in \eqref{a_pde} are added for the control design but don't have any incidence on the stability of the target system.
\\
\begin{lemma} The zero equilibrium of \eqref{a_pde},\eqref{b_pde} with boundary conditions \eqref{IC_ab} and initial conditions $(\alpha^0, \beta^0) \in  \mathcal{L}^2([0,1])$ is exponentially stable in the $\mathcal{L}^2$ sense
\end{lemma}
\begin{proof}
Consider the following candidate Lyapunov functional :
\begin{align}
V(t) = \int^{1}_{0}\left(e^{-\delta x} \sum^{n}_{i=1} \dfrac{ \alpha_i(t,x)^2}{\lambda_i} + le^{\delta x} \sum^{n}_{i=1} \dfrac{ \beta_i(t,x)^2}{\mu_i}\right)dx 
\end{align}
where  $l>0$ and $\delta>0$  are parameters to be determined. One should notice that $\sqrt{V}$ is equivalent to the $\mathcal{L}^2$ norm. After differentiating $V$ with respect to time and integrating by part we get :
\begin{multline}
\dot{V}(t)=[-e^{-\delta x}\alpha(t,x)^T\alpha(t,x) + le^{\delta  x} \beta(t,x)^T \beta(t,x)]^1_0 \\
- \int^{1}_{0}\delta e^{-\delta x}\alpha(t,x)^T \alpha(t,x) dx - \int^{1}_0 l \delta  e^{\delta  x}\beta(t,x)^T \beta(t,x) dx \\
+2\int^1_0e^{-\delta x}\alpha(t,x)^T (\Lambda^+)^{-1}\Sigma^{++}\alpha(x,t)dx \\
+2\int^1_0e^{-\delta x}\alpha(t,x)^T(\Lambda^+)^{-1}\Sigma^{+-}\beta(t,x)dx \\
+2\int^1_0\int^x_0 e^{-\delta x}\alpha(t,x)^T(\Lambda^+)^{-1}C^+(x, \xi)\alpha(t,\xi)d\xi dx \\
+2\int^1_0\int^x_0 e^{-\delta x}\alpha(t,x)^T(\Lambda^+)^{-1}C^{-}(x,\xi)\beta(t,\xi) d\xi dx \\
+2l\int^1_0 e^{\delta  x} \beta(t,x)^T(\Lambda^-)^{-1}\Omega(x)\beta(t,x) dx
\end{multline}
Let $M>0$, $||q||$ and $\epsilon >0$ be such that
\begin{align}
&\forall i=1,\dots, n\quad  \forall j=1, \dots, n \quad \forall k=1, \dots, m  \nonumber \\
&\forall l=1, \dots, m \quad \forall x \in [0,1]  \quad \Sigma^{++}_{ij}, \Sigma^{+-}_{ik}, C^+_{ij}, \nonumber  \\
& \quad \quad \quad \quad C^{-}_{ik}, \Omega_{kl}(x)< M \nonumber \\
&\forall i=1, \dots, n \quad  \lambda_i > \epsilon \quad \text{and} \quad \mu_i > \epsilon \nonumber\\
& ||q||= \displaystyle \max_{i=1,\dots,m, j=1,\dots,m} q_{ij} \nonumber
\end{align}
Using Young's and Cauchy-Schwarz inequalities and the boundary conditions yields
\begin{multline}
\dot{V}(t) \leq \beta(t,0)^T(Q_0^TQ_0-lI_{m\times m})\beta(t,0)\\
-\int^1_0 e^{-\delta x}\alpha(t,x)^TP\alpha(t,x)dx\\
-\int^1_0 le^{\delta  x}\beta(t,x)^TQ(x)\beta(t,x)dx
\end{multline}
with $P=\left((\delta-\dfrac{2mM}{\epsilon}-\dfrac{nM}{\epsilon}-\dfrac{Mn}{\delta\epsilon})I_{n\times n}-2(\Lambda^+)^{-1}\Sigma^{++}\right)$ and $Q(x)=\left(\delta -\dfrac{pnM}{l\epsilon}e^{-\delta  x}-\dfrac{Mn}{l\delta \epsilon}e^{-\delta  x}\right)I_{m \times m} - 2(\Lambda^-)^{-1}\Omega(x)$.\\
Choosing $l$ such that $l>m||q||$ ensures that $\beta(t,0)^T(Q_0^TQ_0-lI_{m\times m})\beta(t,0) \leq s\beta(t,0)^T\beta(t,0)$ for some $s < 0$. Taking $\delta$ large enough ensures that $Q(x)$ and $P$ are positive definite for all $x \in [0,1]$. 
This concludes the proof
\\
\end{proof}
Besides, the following lemma assesses the finite-time stability of the target system.
\begin{lemma}
The system \eqref{a_pde}, \eqref{b_pde} reaches its zero equilibrium in finite-time $t_F=\dfrac{1}{\mu_1}+\dfrac{1}{\lambda_1}$
\end{lemma}
\begin{proof}
The proof of this lemma is straightforward using the proof of \cite[Lemma 3.1]{hu2015control} \\
\end{proof}

\subsubsection{Volterra Transformation}
In order to map the original system \eqref{u_pde}-\eqref{IC} to the target system \eqref{a_pde}-\eqref{IC_ab}, we use the following Volterra transformation 
\begin{align}
&\alpha(t,x)=u(t,x) \label{back1} \\
&\beta(t,x)=v(t,x) \nonumber \\
&\quad \quad -\int^x_0(K(x,\xi)u(\xi)+L(x,\xi)v(\xi))d\xi \label{back2}
\end{align}
where the kernels $K$ and $L$, defined on ~$\mathcal{T}=\{(x,\xi)\in [0,1]^2 | \xi \leq x\}$ have yet to be defined. Differentiating \eqref{back2} with respect to space and using the Leibniz rule yields 
\begin{multline}
\beta_x(t,x)=v_x(t,x)-K(x,x)u(t,x)-L(x,x)v(t,x)\\
-\int^x_0K_x(x,\xi)u(t,\xi)+L_x(x,\xi)v(t,\xi)d\xi
\end{multline}
Differentiating with respect to time, using \eqref{u_pde}, \eqref{v_pde} and integrating by parts yields
\begin{multline}
\beta_t(t,x)=\Lambda^-v_x(t,x)+\Sigma^{-+}u(t,x)+\Sigma^{--}v(t,x)\\
-\int^x_0\Bigg[K(x,\xi)\Sigma^{++}u(t,\xi)+K(x,\xi)\Sigma^{+-}v(t,\xi)\\
+L(x,\xi)\Sigma^{-+}u(t,\xi)+L(x,\xi)\Sigma^{--}v(t,\xi)\Bigg]d\xi \\
+K(x,x)\Lambda^+u(t,x) -K(x,0)\Lambda^+u(t,0)\\
-L(x,x)\Lambda^-v(t,x)+L(x,0)\Lambda^-v(t,0) \\ 
-\int^x_0\bigl[K_{\xi}(x,\xi)\Lambda^+u(t,\xi)-L_{\xi}(x,\xi)\Lambda^{-}v(t,\xi)\bigr]d\xi
\end{multline}
Plugging those expressions into the target system \eqref{back1}-\eqref{back2}, noticing that $\beta(t,0)=v(t,0)$ and using the corresponding boundary conditions \eqref{IC} yields the following system of kernel equations
\begin{align}
0=&\Sigma^{-+}+K(x,x)\Lambda^{+}+\Lambda^-K(x,x) \label{kernel_1}\\
0=&\Sigma^{--}+\Lambda^-L(x,x)-L(x,x)\Lambda^--\Omega(x)\label{kernel_2}\\
0=&K(x,0)\Lambda^+Q_0-L(x,0)\Lambda^{-} \label{kernel_3}\\
0=&\Lambda^-K_x(x,\xi)-K_{\xi}(x,\xi)\Lambda^+ -K(x,\xi)\Sigma^{++} \nonumber \\
&-L(x,\xi)\Sigma^{-+}+\Omega(x)K(x,\xi) \label{kernel_4}\\
0=&\Lambda^-L_x(x,\xi)+L_{\xi}(x,\xi)\Lambda^- -L(x,\xi)\Sigma^{--} \nonumber \\
&-K(x,\xi)\Sigma^{+-}+\Omega(x)L(x,\xi) \label{kernel_5}
\end{align}
We get the following equations for $C^-(x,\xi)$ and $C^+(x,\xi)$ 
\begin{align}
&C^-(x,\xi)=\Sigma^{+-}L(x,\xi)+\int^{x}_{\xi}C^{-}(x,s)L(s,\xi)ds \label{VolterraC} \\
&C^+(x,\xi)=\Sigma^{+-}K(x,\xi)+\int^{x}_{\xi}C^{-}(x,s)K(s,\xi)ds \label{eqC+}
\end{align}

\begin{remark} One can notice that for each $x\in [0,1]$, equation \eqref{VolterraC} is a Volterra equation on $[0,x]$ where $C^-(x,\cdot)$ is the unknown. Assuming that $K$ and $L$ are well defined and bounded, so is $C^-$. Using \eqref{eqC+} yields explicitly $C^+$ as a function of $C^-$ and $K$.\\
\end{remark}
Developing equation \eqref{kernel_1}-\eqref{kernel_5} we get the following set of kernel PDEs :

\underline{ for $1\leq i \leq m$, $1 \leq j \leq n$}
\begin{multline}
\mu_i\partial_xK_{ij}(x,\xi)-\lambda_j\partial_{\xi}K_{ij}(x,\xi)=\sum^n_{k=1}\sigma_{kj}^{++}K_{ik}(x,\xi) \\
+\sum^m_{p=1}\sigma^{-+}_{pj}L_{ip}(x,\xi)-\sum_{i\leq p\leq m}K_{pj}(x,\xi)\omega_{ip}(x) \label{eq_K}
\end{multline}

\underline{ for $1\leq i \leq m$, $1 \leq j \leq m$}
\begin{multline}
\mu_i\partial_xL_{ij}(x,\xi)+\mu_j\partial_{\xi}L_{ij}(x,\xi)=\sum^m_{k=1}\sigma_{kj}^{--}L_{ik}(x,\xi) \\
+\sum^n_{p=1}\sigma^{+-}_{pj}K_{ip}(x,\xi)-\sum_{i\leq p\leq m}L_{pj}(x,\xi)\omega_{ip}(x) \label{eq_L}
\end{multline}
with the following set of boundary conditions 
\begin{multline}
\forall 1 \leq i \leq m, \forall 1 \leq j \leq n, \quad  K_{ij}(x,x)=-\dfrac{\sigma_{ij}^{-+}}{\mu_i+\lambda_j} = k_{ij} \label{eq_K1}
\end{multline}

\begin{multline}
\forall 1 \leq i,j \leq m, j<i \quad L_{ij}(x,x)=\dfrac{-\sigma_{ij}^{--}}{\mu_i-\mu_j} \label{eq_L1}
\end{multline}
\begin{multline}
\forall 1 \leq i,j \leq m, \quad  \mu_jL_{ij}(x,0)=\sum^n_{k=1}\lambda_kK_{ik}(x,0)q_{kj} \label{eq_L2}
\end{multline}
Besides, \eqref{kernel_2} imposes 
\begin{align}
\forall  i \leq j \quad \omega_{ij}(x) = (\mu_i- \mu_j)L_{ij}(x,x)+\sigma_{ij}^{--} \label{eq_Omega}
\end{align}
This induces a coupling between the kernels through equations \eqref{eq_K} and \eqref{eq_L} that could appear as non linear at first sight. However, as it will appear in the proof of the following theorem, the coupling has a linear cascade structure. More precisely, the well-posedness of the target system is assessed in the following theorem.
\begin{theorem}
Consider system \eqref{eq_K}-\eqref{eq_L2}. There exists a unique solution $K$ and $L$ in $L^{\infty}(\mathcal{T})$. 
\end{theorem}
The proof of this theorem is described in the following section and uses the cascade structure of the kernel equations (which is due to the particular shape of the matrix $\Omega$).

\section{Well-posedness of the kernel equation}
To prove the well-posedness of the kernel equations we classically (see \cite{john1960continuous} and \cite{whitham2011linear}) transform the kernel equations into integral equations and use the method of successive approximations. 
\\
\\
By induction, let us consider the following property $P(s)$ defined for all $1 \leq s \leq m$ : \\
$\forall \quad m+1-s \leq i \leq m$ the problem \eqref{eq_K}-\eqref{eq_L2} where $\Omega$ is defined by \eqref{eq_Omega} has a unique solution $K, L \in L^{\infty}(\mathcal{T})$. \\
\\
\underline{Initialization :} For $s=1$, system \eqref{eq_K}-\eqref{eq_L2} rewrites as follow \\
\underline{ for $1 \leq j \leq n$}
\begin{multline}
\mu_m\partial_xK_{mj}-\lambda_j\partial_{\xi}K_{mj}=\sum^n_{k=1}\sigma_{kj}^{++}K_{mk}(x,\xi) \\
+\sum^m_{p=1}\sigma^{-+}_{pj}L_{mp}(x,\xi)-K_{mj}(x,\xi)\sigma_{mm}^{--}
\end{multline}
\underline{ for $1 \leq j \leq m$}
\begin{multline}
\mu_m\partial_xL_{mj}+\mu_j\partial_{\xi}L_{mj}=\sum^m_{k=1}\sigma_{kj}^{--}L_{mk}(x,\xi) \\
+\sum^n_{p=1}\sigma^{+-}_{pj}K_{mp}(x,\xi)-L_{mj}(x,\xi)\sigma_{mm}^{--}
\end{multline}
with the following set of boundary conditions 
\begin{multline}
\forall 1 \leq j \leq n, \quad  K_{mj}(x,x)=-\dfrac{\sigma_{mj}^{-+}}{\mu_m+\lambda_j} = k_{1j}
\end{multline}
\begin{multline}
\forall 1 \leq j < m, \quad L_{mj}(x,x)=-\dfrac{\sigma_{mj}^{--}}{\mu_m-\mu_j}
\end{multline}
\begin{multline}
\forall 1 \leq j \leq m, \quad  \mu_jL_{1j}(x,0)=\sum^n_{k=1}\lambda_kK_{1k}(x,0)q_{kj} 
\end{multline}
The well-posedness of such system has been proved in \cite{di2013stabilization}. \\
\\
\underline{Induction :} Let us assume that the property $P(s-1)$ ($1<s \leq m-1$) is true. We consequently have that $\forall \quad m+2-s \leq p \leq m$, $\forall \quad 1 \leq j \leq n$, $\forall \quad 1 \leq l \leq m$ $K_{pj}(\cdot,\cdot)$ and $L_{pl}(\cdot,\cdot)$ are bounded. In the following we take $i=m+1-s$. We now show that \eqref{eq_K}-\eqref{eq_L2} is well-posed and that  $K_{ij}(\cdot,\cdot)$ and $L_{il}(\cdot,\cdot)$ $\in L^{\infty}(\mathcal{T})$

\subsection{Method of characteristics}
\subsubsection{Characteristics of the K kernels} For each $1\leq j \leq n$ and ($x, \xi$) $\in \mathcal{T}$, we define the following characteristic lines ($x_{ij}(x,\xi,\cdot), \xi_{ij}(x,\xi,\cdot)$) corresponding to equation \eqref{eq_K} 

\begin{align}
\left\{
    \begin{array}{l}
        \dfrac{dx_{ij}}{ds}(x,\xi,s) = -\mu_i \quad  s \in [0,s_{ij}^F(x,\xi)]\\
        x_{ij}(x,\xi,0)=x, \quad x_{ij}(x,\xi,s_{ij}^F(x,\xi))=x_{ij}^F(x,\xi) 
\end{array}
\right.
\label{char_K_1}
\end{align}
\begin{align}
\left\{
    \begin{array}{l}
        \dfrac{d\xi_{ij}}{ds}(x,\xi,s) = \lambda_j \quad  s \in [0,s_{ij}^F(x,\xi)]\\
        \xi_{ij}(x,\xi,0)=\xi, \quad \xi_{ij}(x,\xi,s_{ij}^F(x,\xi))=x_{ij}^F(x,\xi) 
\end{array}
\right.
\label{char_K_2}
\end{align}

These lines originate at the point $(x,\xi)$ and terminate on the hypothenuse at the point ($x_{ij}^F(x,\xi), x_{ij}^F(x,\xi))$. Integrating \eqref{eq_K} along these characteristics and using the boundary conditions \eqref{eq_K1} we get 
\begin{multline}
K_{ij}(x,\xi)=k_{ij}\\
+\int^{s_{ij}^F(x,\xi)}_0\bigl[\sum^n_{k=1}\sigma_{kj}^{++}K_{ik}(x_{ij}(x,\xi,s),\xi_{ij}(x,\xi,s))\\
+\sum^m_{k=1}\sigma_{kj}^{-+}L_{ik}(x_{ij}(x,\xi,s),\xi_{ij}(x,\xi,s))\\
-\sum_{i \leq p \leq m }K_{pj}(x_{ij}(x,\xi,s),\xi_{ij}(x,\xi,s)) \\
((\mu_i-\mu_p)L_{ip}(x_{ij}(x,\xi,s),x_{ij}(x,\xi,s))+\sigma_{ip}^{--})\bigr]ds
\label{eqIntK}
\end{multline}
We can notice that the last sum uses the expression of $K_{pj}$ for $i\leq p \leq m$. This term is known and bounded for $p>i$ (hypothesis of induction). For $p=i$, $\mu_i=\mu_p$ and the term $(\mu_i-\mu_p)L_{ip}(x_{ij}(x,\xi,s),x_{ij}(x,\xi,s)$ cancels.
\subsubsection{Characterisitcs of the L kernels} For each $1\leq j \leq n$ and ($x, \xi$) $\in \mathcal{T}$, we define the following characteristic lines ($\chi_{ij}(x,\xi,\cdot), \zeta_{ij}(x,\xi,\cdot)$) corresponding to equation \eqref{eq_L}

\begin{align}
\left\{
    \begin{array}{l}
        \dfrac{d\chi_{ij}}{d\nu}(x,\xi,s) = -\mu_i \quad  \nu \in [0,\nu_{ij}^F(x,\xi)]\\
        \chi_{ij}(x,\xi,0)=x, \quad \chi_{ij}(x,\xi,\nu_{ij}^F(x,\xi))=\chi_{ij}^F(x,\xi) 
\end{array}
\right.
\label{char_L_1}
\end{align}
\begin{align}
\left\{
    \begin{array}{l}
        \dfrac{d\zeta_{ij}}{d\nu}(x,\xi,s) = -\mu_j \quad  \nu \in [0,\nu_{ij}^F(x,\xi)]\\
        \zeta_{ij}(x,\xi,0)=\xi, \quad \zeta_{ij}(x,\xi,\nu_{ij}^F(x,\xi))=\zeta_{ij}^F(x,\xi) 
\end{array}
\right.
\label{char_L_2}
\end{align}
These lines all originates from $(x,\xi)$ and terminate at the point $(\chi_{ij}^F(x,\xi) ,\zeta_{ij}^F(x,\xi))$, i.e either at $(\chi_{ij}^F(x,\xi),\chi_{ij}^F(x,\xi))$ or at $(\chi_{ij}^F(x,\xi),0)$. Integrating \eqref{eq_L} along these characteristic and using the boundary conditions \eqref{eq_L1}, \eqref{eq_L2} yields
\begin{multline}
L_{ij}(x,\xi)=-\delta_{ij}(x,\xi)\dfrac{\sigma_{ij}^{--}}{\mu_i-\mu_j}\\
+(1-\delta_{ij})\dfrac{1}{\mu_j}\sum_{k=1}^n \lambda_kq_{kj}K_{ik}(\chi_{ij}^F(x,\xi),0)\\
+\int^{\nu_{ij}^F(x,\xi)}_0\bigl[\sum_{p=1}^m \sigma_{pj}^{--}L_{ip}(\chi_{ij}(x,\xi,\nu),\zeta_{ij}(x,\xi,\nu))\\
+\sum_{k=1}^n \sigma_{kj}^{+-}K_{ik}(\chi_{ij}(x,\xi,\nu),\zeta_{ij}(x,\xi,\nu))\\
-\sum_{i \leq p \leq m} L_{pj}(\chi_{ij}(x,\xi,\nu),\zeta_{ij}(x,\xi,\nu))\\
((\mu_i-\mu_p)L_{ip}(\chi_{ij}(x,\xi,\nu),\chi_{ij}(x,\xi,\nu))+\sigma_{ip}^{--})
\bigr]d\nu
\label{eqIntL}
\end{multline}
where the coefficient $\delta_{ij}(x,\xi)$ is defined by
\begin{align}
\delta_{i,j}(x,\xi)=\left\{
    \begin{array}{ll}
        1 & \mbox{if } j < i \quad \mbox{and} \quad \mu_i\xi-\mu_j x \geq 0 \\
        0 & \mbox{else}
            \end{array}
\right.
\end{align}
This coefficient reflects the facts that, as mentioned above, some characteristics terminate on the hypothenuse and others on the axis $\xi=0$. We can now plug \eqref{eqIntK} evaluated at ~$(\chi_{ij}^F(x,\xi),0)$ into \eqref{eqIntL} which yields
\begin{multline}
L_{ij}(x,\xi)=-\delta_{ij}(x,\xi)\dfrac{\sigma_{ij}^{--}}{\mu_i-\mu_j}\\
+(1-\delta_{ij})\dfrac{1}{\mu_j}\sum_{k=1}^n \lambda_kq_{kj}k_{ik}
+(1-\delta_{ij})\dfrac{1}{\mu_j}\sum_{r=1}^n \lambda_rq_{rj}\int_0^{s_{ir}^F(\chi_{ij}^F(x,\xi),0)} \\
\bigr[\sum^n_{k=1}\sigma_{kr}^{++}K_{ik}(x_{ir}(\chi_{ij}^F(x,\xi),0,s),\xi_{ir}(\chi_{ij}^F(x,\xi),0,s))\\
+\sum^m_{k=1}\sigma_{kr}^{-+}L_{ik}(x_{ir}(\chi_{ij}^F(x,\xi),0,s),\xi_{ir}(\chi_{ij}^F(x,\xi),0,s))\\
-\sum_{i \leq p \leq m}K_{pr}(x_{ir}(\chi_{ij}^F(x,\xi),0,s),\xi_{ir}(\chi_{ij}^F(x,\xi),0,s))\\
((\mu_i-\mu_p)L_{ip}(x_{ij}((\chi_{ij}^F(x,\xi),0,s),x_{ij}((\chi_{ij}^F(x,\xi),0,s))+\sigma_{ip}^{--})\bigr]ds\\
+\int^{\nu_{ij}^F(x,\xi)}_0\bigl[\sum_{p=1}^m \sigma_{pj}^{--}L_{ip}(\chi_{ij}(x,\xi,\nu),\zeta_{ij}(x,\xi,\nu))\\
+\sum_{k=1}^n \sigma_{kj}^{+-}K_{ik}(\chi_{ij}(x,\xi,\nu),\zeta_{ij}(x,\xi,\nu))\\
-\sum_{i \leq p \leq m} L_{pj}(\chi_{ij}(x,\xi,\nu),\zeta_{ij}(x,\xi,\nu))\\
((\mu_i-\mu_p)L_{ip}(\chi_{ij}(x,\xi,\nu),\chi_{ij}(x,\xi,\nu))+\sigma_{ip}^{--})
\bigr]d\nu
\label{eqInt}
\end{multline}

\subsection{Method of successive approximations}
 In order to solve the integral equations \eqref{eqIntK}, \eqref{eqInt} we use the method of successive approximations. We define 
 \begin{align}
 &\forall 1 \leq j \leq n \quad \phi_j^1(x,\xi)=k_{ij} \nonumber\\
&-\int_0^{s_{ij}^F(x,\xi)}\sum_{i<p\leq m}K_{pj}(x_{ij}(x,\xi,s),\xi_{ij}(x,\xi,s))\sigma_{ip}^{--}  \\
 &\forall 1 \leq j \leq  m \quad \phi_j^2(x,\xi)=-\delta_{ij}(x,\xi)\dfrac{\sigma_{ij}^{--}}{\mu_i-\mu_j} \nonumber \\
&+(1-\delta_{ij})\dfrac{1}{\mu_j}\sum_{k=1}^n \lambda_kq_{kj}k_{ik} \nonumber \\
&-(1-\delta_{ij})\dfrac{1}{\mu_j}\sum_{r=1}^n \lambda_rq_{rj}\int_0^{s_{ir}^F(\chi_{ij}^F(x,\xi),0)} \nonumber \\
&\sum_{i < p \leq m}K_{pr}(x_{ir}(\chi_{ij}^F(x,\xi),0,s),\xi_{ir}(\chi_{ij}^F(x,\xi),0,s))\sigma_{ip}^{--} \nonumber \\
&-\int_0^{s_{ij}^F(x,\xi)}\sum_{i<p\leq m}L_{pj}(\chi_{ij}(x,\xi,\nu),\zeta_{ij}(x,\xi,\nu))\sigma_{ip}^{--}
\end{align}
Besides we denote H as the vector containing the kernels 
\begin{align}\
H&=\begin{pmatrix} K_{i1} & \dots & K_{in} & L_{i1} & \cdots &L_{im} \end{pmatrix}^{\top}\\
\Psi&=\begin{pmatrix} \phi^1_{1} & \dots & \phi^1_{n} & \phi^2_{1} & \dots & \phi^2_{m} \end{pmatrix}^{\top}
\end{align}
We now consider the following operators : $\forall 1 \leq j \leq n$
\begin{multline}
 \Phi_j^1(H)(x,\xi)=\\
\int^{s_{ij}^F(x,\xi)}_0\bigl[\sum^n_{k=1}\sigma_{kj}^{++}K_{ik}(x_{ij}(x,\xi,s),\xi_{ij}(x,\xi,s))\\
+\sum^m_{k=1}\sigma_{kj}^{-+}L_{ik}(x_{ij}(x,\xi,s),\xi_{ij}(x,\xi,s))\\
-\sum_{i < p \leq m }K_{pj}(x_{ij}(x,\xi,s),\xi_{ij}(x,\xi,s)) \\
((\mu_i-\mu_j)L_{ip}(x_{ij}(x,\xi,s),x_{ij}(x,\xi,s)))\\
+\sigma_{ii}^{--}K_{ij}(x_{ij}(x,\xi,s),\xi_{ij}(x,\xi,s))\bigr]ds
\end{multline}
$\forall 1 \leq j \leq m$
\begin{multline}
 \Phi_j^2(H)(x,\xi)=\\
(1-\delta_{ij})\dfrac{1}{\mu_j}\sum_{r=1}^n \lambda_rq_{rj}\int_0^{s_{ir}^F(\chi_{ij}^F(x,\xi),0)} \\
\bigr[\sum^n_{k=1}\sigma_{kr}^{++}K_{ik}(x_{ir}(\chi_{ij}^F(x,\xi),0,s),\xi_{ir}(\chi_{ij}^F(x,\xi),0,s))\\
+\sum^m_{k=1}\sigma_{kr}^{-+}L_{ik}(x_{ir}(\chi_{ij}^F(x,\xi),0,s),\xi_{ir}(\chi_{ij}^F(x,\xi),0,s))\\
-\sum_{i < p \leq m}K_{pr}(x_{ir}(\chi_{ij}^F(x,\xi),0,s),\xi_{ir}(\chi_{ij}^F(x,\xi),0,s))\\
((\mu_i-\mu_p)L_{ip}(x_{ij}((\chi_{ij}^F(x,\xi),0,s),x_{ij}((\chi_{ij}^F(x,\xi),0,s))) \\
-K_{ir}(x_{ir}(\chi_{ij}^F(x,\xi),0,s),\xi_{ir}(\chi_{ij}^F(x,\xi),0,s))\sigma_{ii}^{--}\bigr]ds\\
+\int^{\nu_{ij}^F(x,\xi)}_0\bigl[\sum_{p=1}^m \sigma_{pj}^{--}L_{ip}(\chi_{ij}(x,\xi,\nu),\zeta_{ij}(x,\xi,\nu))\\
+\sum_{k=1}^n \sigma_{kj}^{+-}K_{ik}(\chi_{ij}(x,\xi,\nu),\zeta_{ij}(x,\xi,\nu))\\
-\sum_{i < p \leq m} L_{pj}(\chi_{ij}(x,\xi,\nu),\zeta_{ij}(x,\xi,\nu))\\
((\mu_i-\mu_p)L_{ip}(\chi_{ij}(x,\xi,\nu),\chi_{ij}(x,\xi,\nu))) \\
- L_{ij}(\chi_{ij}(x,\xi,\nu),\zeta_{ij}(x,\xi,\nu))\sigma_{ii}^{--}
\bigr]d\nu
\end{multline}
We set $\Phi[H](x,\xi)=[\Phi^1[H](x,\xi)^T,\Phi^2[H](x,\xi)^{\top}]^{\top}$
We define the following sequence 
\begin{align}
&H^0(x,\xi)=0 \\
&H^q(x,\xi)=\Psi(x,\xi)+\Phi(H^{q-1})(x,\xi)
\end{align}
Consequently, if the sequence $H^q$ has a limit, then this limit is a solution of the integral equation and therefore of the original system. \\
We define the increment $\Delta H^q=H^q-H^{q-1}$ (with $\Delta H^0=\Psi$). Provided the limit exists one has
\begin{align}
H(x,\xi)=\displaystyle \lim_{q \rightarrow + \infty} H^q(x,\xi)= \sum_{q=0}^{+ \infty}\Delta H^q(x,\xi)
\label{SUM}
\end{align}
We now prove the convergence of the series.
\begin{remark}
The proof of the convergence of the success approximations dries is similar to the one given in \cite{di2013stabilization}, since all the characteristic lines have the same direction along the $x-$axis.
\end{remark}
\subsection{Convergence of the successive approximation series}
Similarly to \cite{di2013stabilization, hu2015control} we want to find a recursive upper bound in order to prove the convergence of the series.
We first define 
\begin{align}
&\bar{\Phi}=\displaystyle \max_{j} \displaystyle \max_{(x,\xi)\in \mathcal{T}} \{ |\phi^1_{i,j}(x,\xi)|, | \phi^2_{ij}(x,\xi)| \} \\
&\bar{\sigma}=\displaystyle \max_{k,j}\{\sigma^{++}_{kj},\sigma^{+-}_{kj},\sigma^{-+}_{kj},\sigma^{--}_{kj}\}, \quad \bar{q}=\displaystyle \max_{k,j}\{q_{kj}\} \nonumber \\
&\bar{\mu}=\displaystyle \max_{p}\{|\mu_i-\mu_p|\}, \quad \bar{\lambda}=\max\{\lambda_n, \mu_n\} \nonumber \\
&\tilde{\lambda}=\max\{\dfrac{1}{\lambda_1},\dfrac{1}{\mu_1}\}, \quad \nonumber \\
&M_{\lambda}= \displaystyle \max_{j=1,\dots,m}\{\dfrac{1}{\mu_j}\}\nonumber
\end{align}
We then define $\bar{S}=\displaystyle \max_{p>i, 1\leq j \leq n} \{||K_{pj}||, ||L_{pj}\}$ which is well defined according to the hypothesis $P(s-1)$. Moreover we set 
\begin{align}
M=(n\bar{\lambda}\tilde{\lambda}\bar{q}+1)[(n+m+1)\bar{\sigma}+m\bar{\mu}\bar{S}]M_{\lambda}
\end{align}
We recall the following result from \cite[Lemma 5.5]{di2013stabilization}
\\
\begin{lemma}
For any integer $q$, $(x,\xi) \in \mathcal{T}$ and $s_{ij}^F(x,\xi), \nu^F_{ij}(x,\xi), x_{ij}(x,\xi,\cdot), \xi_{ij}(x,\xi, \cdot), \chi_{ij}(x,\xi,\cdot), \zeta_{ij}(x,\xi,\cdot)$ defined as in \eqref{char_K_1}, \eqref{char_K_2}, \eqref{char_L_1}, \eqref{char_L_2} respectively, the following inequalities holds 
\begin{align}
\forall 1 \leq k \leq m, &\quad \forall 1 \leq j \leq n \nonumber \\
&\int_0^{s_{kj}^F(x,\xi)}(x_{kj}(x,\xi,s))^qds \leq M_{\lambda}\dfrac{x^{q+1}}{q}  \label{ineq_1} \\
\forall 1 \leq k \leq m, &\quad \forall 1 \leq j \leq n \nonumber \\
&\int_0^{\nu_{kj}^F(x,\xi)}(\chi_{ij}(x,\xi,s))^qds \leq M_{\lambda}\dfrac{x^{q+1}}{q} \label{ineq_2} 
\end{align}  
\end{lemma}
\begin{lemma}
Assume that for some $1\leq q$, one has, for all ~$(x, \xi) \in \mathcal{T}$
\begin{align}
\forall j=1,...m+n \quad |\Delta H^q_j(x,\xi)| \leq \bar{\Phi}\dfrac{M^qx^{q}}{q!}
\label{hyp_rec}
\end{align}
where $\Delta H^q_j(x,\xi)$ is the $j$-th component of $\Delta H^q(x,\xi)$.\\
Then, one has
\begin{align}
\forall j=1,...m+n \quad |\Delta H^{q+1}_j(x,\xi)| \leq \bar{\Phi}\dfrac{M^{q+1}x^{q+1}}{(q+1)!}
\end{align}
\end{lemma}
\begin{proof}
Assume that \eqref{hyp_rec} holds for some fixed $1\leq q$. Let us consider $1 \leq j \leq (m+n)$.\\
\\
\underline{Case $j \leq n$}
\begin{multline}
|\Delta H_j^{q+1}| =\\
 |\int^{s_{ij}^F(x,\xi)}_0\bigl[\sum^n_{k=1}\sigma_{kj}^{++}\Delta K^q_{ik}(x_{ij}(x,\xi,s),\xi_{ij}(x,\xi,s))\\
+\sum^m_{k=1}\sigma_{kj}^{-+}\Delta L^q_{ik}(x_{ij}(x,\xi,s),\xi_{ij}(x,\xi,s))\\
-\sum_{i < p \leq m}K_{pj}(x_{ij}(x,\xi,s),\xi_{ij}(x,\xi,s))\\
\cdot(\mu_i-\mu_p)\Delta L_{ip}^q(x_{ij}(x,\xi,s),x_{ij}(x,\xi,s)) \\
-\Delta K^q_{ij}(x_{ij}(x,\xi,s),\xi_{ij}(x,\xi,s))\sigma_{ii}^{--}
\bigr]ds|
\end{multline}
Consequently, using \eqref{hyp_rec} and \eqref{ineq_1}
\begin{align}
|\Delta H_j^{q+1}| &\leq \int^{s_{ij}^F(x,\xi)}_0 ((n+m+1)\bar{\sigma}+m\bar{S}\bar{\mu})\\
&\quad\cdot \bar{\Phi}\dfrac{M^q(x_{ij}(x,\xi,s))^q}{q!}ds \nonumber\\
& \leq ((n+m+1)\bar{\sigma}+m\bar{S}\bar{\mu})\dfrac{\bar{\Phi}M^q}{q!}M_{\lambda}\dfrac{x^{q+1}}{q+1}\nonumber \\
& \leq \bar{\Phi}\dfrac{M^{q+1}x^{q+1}}{(q+1)!}
\end{align}
\underline{Case $n < j \leq n+m$}
Using \eqref{hyp_rec} we get 
\begin{align}
|\Delta H_j^{q+1}| &\leq \bar{\lambda}\tilde{\lambda}\bar{q}((n+m+1)\bar{\sigma}+m\bar{S}\bar{\mu})\sum_{r=1}^n \nonumber \\
&\cdot \int_0^{s_{ir}^F(\chi_{ij}^F(x,\xi),0)}\bar{\phi}\dfrac{M^q(x_{ir}(\chi_{ij}^F(x,\xi),0,s))^q}{q!} \nonumber \\
&+((n+m+1)\bar{\sigma}+m\bar{S}\bar{\mu})\int_0^{\nu_{ij}^F(x,\xi)}\bar{\Phi}\dfrac{M^q(\chi_{ij})^q}{q!}d\nu \nonumber \\
& \leq (n\tilde{\lambda}\bar{\lambda}\bar{q}+1)((n+m+1)\bar{\sigma}+m\bar{S}\bar{\mu})\bar{\Phi}M_{\lambda}\dfrac{M^qx^{q+1}}{(q+1)!} \nonumber \\
& \leq \bar{\Phi}\dfrac{M^{q+1}x^{q+1}}{(q+1)!}
\end{align}
This concludes the proof
\end{proof}
Consequently, using similar procedures that the ones presented in \cite{di2013stabilization,vazquez2011local}, we get that \eqref{SUM} converges and thus the property $P(s)$ is true. This concludes the proof by induction of Theorem 1. 

\section{Control law and main results}
We now state the main stabilization result as follows.
\begin{theorem}
System \eqref{u_pde}-\eqref{v_pde} with boundary conditions \eqref{IC} and the following feedback control law 
\begin{align}
U(t) &= -R_1u(t,1) \nonumber \\
+& \int_0^1[K(1,\xi)u(t,\xi)+L(1,\xi)v(t,\xi)]d\xi \label{eq_U}
\end{align}
reaches its zero equilibrium in finite time $t_F=$ where $t_F$ is given by \eqref{t_F}. The zero equilibrium is exponentially stable in the $L^2$-sense.
\end{theorem}
\begin{proof}
Notice first that evaluating \eqref{back2} at $x=1$ yields \eqref{eq_U}. Besides, rewriting \eqref{back2} as follows
\begin{align}
\begin{pmatrix} \alpha(t,x) \\ \beta(t,x) \end{pmatrix} &= \begin{pmatrix} u(t,x) \\ v(t,x) \end{pmatrix} \nonumber \\
&- \int_0^x\begin{pmatrix}0 & 0 \\ K(x,\xi) & L(x,\xi) \end{pmatrix}\begin{pmatrix}u(t,\xi) \\ v(t,\xi) \end{pmatrix} d\xi
\end{align}
It is a classical Volterra equation of the second kind. One can check from \cite{hochstadt2011integral} that there exists a unique function $\mathcal{S}$ such that 
\begin{align}
\begin{pmatrix}u(t,x) \\ v(t,x)\end{pmatrix}=\begin{pmatrix}\alpha(t,x) \\ \beta(t,x) \end{pmatrix} - \int_0^x  \mathcal{S}(x,\xi) \begin{pmatrix} \alpha(t,\xi) \\ \beta(t,\xi) \end{pmatrix}
\end{align}
Applying Lemma 2 implies that $(\alpha, \beta)$ go to zero in finite time $t_F$, therefore $(u,v)$ converge to zero in finite time
\end{proof}
\begin{remark}
The time of convergence $t_F$ is smaller than the one given in \cite{hu2015control}. Nevertheless we have lost here some degrees of freedom in the kernel equations and thus in the controller gains.
\end{remark}
\section{Uncollocated observer design and output feedback controller}
In this section we design an observer that relies on the measurements of $v$ at the left boundary, i.e we measure 
\begin{align}
y(t)=v(t,0)
\end{align}
Then, using the estimates given by our observer and the control law \eqref{eq_U}, we derive an output feedback controller.

\subsection{Observer design}
The observer equations read as follows
\begin{align}
\hat{u}_t(t,x)+\Lambda^+\hat{u}_x(t,x)=&\Sigma^{++}\hat{u}(t,x)+\Sigma^{+-}\hat{v}(t,x) \nonumber \\
& -P^+(x)(\hat{v}(t,0)-v(t,0)) \label{hat_u}\\
\hat{v}_t(t,x)+\Lambda^-\hat{v}_x(t,x)=&\Sigma^{-+}\hat{u}(t,x)+\Sigma^{--}\hat{v}(t,x) \nonumber \\
& -P^-(x)(\hat{v}(t,0)-v(t,0))  \label{hat_v}
\end{align}
with the boundary conditions 
\begin{align}
\hat{u}(t,0)=Q_0v(t,0), \quad \hat{v}(t,1)=R_1\hat{u}(t,1) +U \label{hat_boundary}
\end{align}
where $P^+(\cdot)$ and $P^-(\cdot)$ have yet to be designed. This yield the following error system
\begin{align}
\tilde{u}_t(t,x)+\Lambda^+\tilde{u}_x(t,x)=&\Sigma^{++}\tilde{u}(t,x)+\Sigma^{+-}\tilde{v}(t,x) \nonumber \\
& -P^+(x)\tilde{v}(t,0) \label{tilde_u}\\
\tilde{v}_t(t,x)+\Lambda^-\tilde{v}_x(t,x)=&\Sigma^{-+}\tilde{u}(t,x)+\Sigma^{--}\tilde{v}(t,x) \nonumber \\
& -P^-(x)\tilde{v}(t,0) \label{tilde_v}
\end{align}
with the boundary conditions 
\begin{align}
\tilde{u}(t,0)=0, \quad \tilde{v}(t,1)=R_1\tilde{u}(t,1) \label{tilde_boundary}
\end{align}

\subsection{Target system}

We map the system \eqref{tilde_u}-\eqref{tilde_boundary} to the following system 
\begin{align}
\tilde{\alpha}_t(t,x) + \Lambda^+ \tilde{\alpha}_x(t,x) &= \Sigma^{++} \tilde{\alpha}(t,x) \nonumber \\
&+\int_0^xD^+(x,\xi)\tilde{\alpha}(t,\xi)d\xi \label{tilde_a} \\
\tilde{\beta}_t(t,x)-\Lambda^-\tilde{\beta}_x(t,x)&=\Sigma^{-+}\tilde{\alpha}(t,x)+\Omega(x)\beta(t,x) \nonumber \\
&+ \int_0^xD^-(x,\xi)\tilde{\alpha}(t,\xi)d\xi \label{tilde_b}
\end{align}
with the following boundary conditions
\begin{align}
\tilde{\alpha}(t,0)=0, \quad \tilde{\beta}(t,1)=R_1\tilde{\alpha}(t,1) \label{tilde_boundary2}
\end{align}
where $D^+$, and $D^-$ are $L^{\infty}$ matrix functions of the domain $\mathcal{T}$ and $\Omega \in L^{\infty}(0,1)$ is an upper triangular matrix with the following structure 
\begin{align}
\Omega(x)= \begin{pmatrix} \omega_{1,1}(x) & \omega_{1,2}(x) & \dots & \omega_{1,m}(x) \\ 0 & \ddots & \ddots & \vdots \\  \vdots & \ddots  & \omega_{m-1,m-1}(x) &\omega_{m-1,m}(x)  \\ 0 &\dots & 0 & \omega_{m,m}(x) \end{pmatrix}
\end{align}
\begin{lemma}
The system \eqref{tilde_a}, \eqref{tilde_b} reaches its zero equilibrium in a finite time $t_F$ where $t_F$ is defined by \eqref{t_F}
\end{lemma}
\begin{proof}
The system is a cascade of $\tilde{\alpha}$-system (that has zero input at the led boundary) into the $\beta$-system (that has zero input at the right boundary once $\tilde{\alpha}$ becomes null).
The rigorous proof of the lemma follows the same step of the proof of Lemma 2 and is omitted here.
\end{proof}

\subsection{Volterra Transformation}
In order to map the original system \eqref{tilde_u}-\eqref{tilde_boundary} to the target system \eqref{tilde_a}-\eqref{tilde_boundary2}, we use the following Volterra transformation
\begin{align}
\tilde{u}(t,x)=\tilde{\alpha}(t,x)+\int_0^xM(x,\xi)\tilde{\beta}(t,\xi)d\xi \label{eq_M0} \\
\tilde{v}(t,x)=\tilde{\beta}(t,x)+\int_0^xN(x,\xi)\tilde{\beta}(t,\xi)d\xi \label{eq_N0}
\end{align}
where the kernels $M$ and $N$ defined on ~$\mathcal{T}=\{(x,\xi)\in [0,1]^2 | \xi \leq x\}$ have yet to defined. Differentiating \eqref{eq_M0}, \eqref{eq_N0} with respect to space and time yields the following kernel equations\\
\underline{ for $1\leq i \leq n$, $1 \leq j \leq m$}
\begin{multline}
\lambda_i\partial_xM_{ij}(x,\xi)-\mu_j\partial_{\xi}M_{ij}(x,\xi)=\sum^n_{k=1}\sigma_{ik}^{++}M_{kj}(x,\xi) \\
+\sum^m_{p=1}\sigma^{-+}_{ip}N_{pj}(x,\xi)-\sum_{p=1}^mM_{ip}(x,\xi)\omega_{pj}(x) \label{eq_M}
\end{multline}

\underline{ for $1\leq i \leq m$, $1 \leq j \leq m$}
\begin{multline}
\mu_i\partial_xN_{ij}(x,\xi)+\mu_j\partial_{\xi}N_{ij}(x,\xi)=-\sum^n_{k=1}\sigma_{ik}^{--}N_{kj}(x,\xi) \\
-\sum^n_{p=1}\sigma^{+-}_{ip}M_{pj}(x,\xi)+\sum_{p=1}^{m}N_{ip}(x,\xi)\omega_{pj}(x) \label{eq_N}
\end{multline}
with the following set of boundary conditions :\\
\begin{multline}
\forall 1 \leq i \leq m, \forall 1 \leq j \leq n, \quad  M_{ij}(x,x)=-\dfrac{\sigma_{ij}^{+-}}{\mu_j+\lambda_i} = k_{ij} \label{eq_M1}
\end{multline}
\begin{multline}
\forall 1 \leq i,j \leq m, j< i \quad N_{ij}(x,x)=\dfrac{-\sigma_{ij}^{--}}{\mu_j-\mu_i} \label{eq_N1}
\end{multline}
\begin{align}
\forall  i \leq j \quad \omega_{ij}(x) = (\mu_j- \mu_i)N_{ij}(x,x)+\sigma_{ij}^{--} \label{eq_Omega2}
\end{align}
Evaluating \eqref{eq_M0}, \eqref{eq_N0} at $x=1$ yields
\begin{multline}
\forall 1 \leq i,j \leq m, \quad  N_{ij}(1,\xi)=\sum^n_{k=1}\rho_{ik}M_{kj}(1,\xi) \label{eq_N2}
\end{multline}

while $d_{ij}^+, d_{ij}^-$ are given by 
\begin{align}
d_{ij}^+(x,,\xi)=&-\sum_{k=1}^mM_{ik}(x,\xi)\sigma_{kj}^{-+} \nonumber \\
&+\int_\xi^x\sum_{k=1}^mM_{ik}(x,s)d_{kj}^-(s,\xi)ds \\
d_{ij}^-(x,,\xi)=&-\sum_{k=1}^mN_{ik}(x,\xi)\sigma_{kj}^{-+} \nonumber \\
&+\int_\xi^x\sum_{k=1}^mN_{ik}(x,s)d_{kj}^-(s,\xi)ds
\end{align}
provided the $M$ and $N$ kernels are well-defined. Finally the observer gains are given by 
\begin{align}
p_{ij}^+(x)=\mu_jM_{ij}(x,0) \\
p_{ij}^-(x)=\mu_jN_{ij}(x,0) 
\end{align}
Considering the following alternate variables
\begin{align}
&\bar{M}_{ij}(\chi,y)=M_{ij}(1-y,1-\chi)=M_{ij}(x,\xi) \\
&\bar{N}_{ij}(\chi,y)=N_{ij}(1-y,1-\chi)=N_{ij}(x,\xi) \\
&\bar{\omega}_{ij}(\chi)=\omega_{ij}(x)
\end{align}
yields \\
\underline{ for $1\leq i \leq n$, $1 \leq j \leq m$}
\begin{multline}
-\lambda_i\partial_{\chi}\bar{M}_{ij}(\chi,y)+\mu_j\partial_{y}\bar{M}_{ij}(\chi,y)=-\sum^n_{k=1}\sigma_{ik}^{++}\bar{M}_{kj}(\chi,y) \\
-\sum^m_{p=1}\sigma^{-+}_{ip}\bar{N}_{pj}(\chi,y)+\sum_{p=1}^m\bar{M}_{ip}(\chi,y)\bar{\omega}_{pj}(\chi) \label{eq_M'}
\end{multline}

\underline{ for $1\leq i \leq m$, $1 \leq j \leq m$}
\begin{multline}
\mu_i\partial_{\chi}\bar{N}_{ij}(\chi,y)+\mu_j\partial_{y}\bar{N}_{ij}(\chi,y)=\sum^n_{k=1}\sigma_{ik}^{--}\bar{N}_{kj}(\chi,y) \\
\sum^n_{p=1}\sigma^{+-}_{ip}\bar{M}_{pj}(\chi,y)-\sum_{p=1}^{m}\bar{N}_{ip}(\chi,y)\bar{\omega}_{pj}(\chi) \label{eq_N'}
\end{multline}
with the following set of boundary conditions 
\begin{multline}
\forall 1 \leq i \leq m, \forall 1 \leq j \leq n, \quad  \bar{M}_{ij}(\chi,\chi)=-\dfrac{\sigma_{ij}^{+-}}{\mu_j+\lambda_i} = k_{ij} \label{eq_M1'}
\end{multline}
\begin{multline}
\forall 1 \leq i,j \leq m, j< i \quad \bar{N}_{ij}(\chi,\chi)=\dfrac{-\sigma_{ij}^{--}}{\mu_j-\mu_i} \label{eq_N1'}
\end{multline}
\begin{align}
\forall  i \leq j \quad \bar{\omega}_{ij}(\chi) = (\mu_j- \mu_i)\bar{N}_{ij}(\chi,\chi)+\sigma_{ij}^{--} \label{eq_Omega2'}
\end{align}
Evaluating \eqref{eq_M0}, \eqref{eq_N0} at $x=1$ yields
\begin{multline}
\forall 1 \leq i,j \leq m, \quad  \bar{N}_{ij}(\chi,0)=\sum^n_{k=1}\rho_{ik}M_{kj}(\chi,0) \label{eq_N2'}
\end{multline}
This system has the same cascade structure as the controller kernel system. Using a similar proof we can asses its well-posedness.

\subsection{Output feedback controller}
The estimates can be used in a observer-controller to derive an output feedback law yielding finite-time stability of the zero equilibrium
\begin{lemma}
Consider the system composed of \eqref{u_pde}-\eqref{IC} and target system \eqref{hat_u}-\eqref{hat_boundary} with the following control law 
\begin{align}
U(t)=\int_0^1[K(1,\xi)\hat{u}(t,\xi)+L(1,\xi)\hat{v}(t,\xi)]d\xi - R_1\hat{u}(t,1)
\end{align}
where $K$ and $L$ are defined by \eqref{eq_K}-\eqref{eq_Omega}. Its solutions ($u,v,\hat{u},\hat{v}$) converge in finite time to zero
\end{lemma}
\begin{proof}
The convergence of the observer error states $\tilde{u}, \tilde{v}$  to zero for $t_F \leq t$ is ensured by Lemma 5, along with the existence of the backstepping transformation. Thus, once $t_F \leq t$, $v(t,0) = \hat{v}(t,0)$ and one can use Theorem2. Therefore for $2t_F \leq t$, one has ($\tilde{u}, \tilde{v}, \hat{u}, \hat{v}$) $\equiv 0$ which yields ($u,v$) $\equiv 0$.
The convergence of the observer error states $\tilde{u}, \tilde{v}$  to zero for $t_F \leq t$ is ensured by Lemma 5, along with the existence of the backstepping transformation. Thus, once $t_F \leq t$, $v(t,0) = \hat{v}(t,0)$ and one can use Theorem2. Therefore for $2t_F \leq t$, one has ($\tilde{u}, \tilde{v}, \hat{u}, \hat{v}$) $\equiv 0$ which yields ($u,v$) $\equiv 0$.
\end{proof}

\section{Simulation results}

In this section we illusttrate our results with simulations on a toy problem. The numerical values of the parameters are as follow. 
\begin{align}
 n=m=2, \quad \mu_1=\lambda_1=1, \quad \mu_2=\lambda_2=2
\end{align}
\begin{align}
\Sigma^{++}&=\Sigma^{+-}=\begin{pmatrix} 1 &0\\ 0 &1 \end{pmatrix} \quad \Sigma^{-+}=\begin{pmatrix} 1 &1\\ 1 &0 \end{pmatrix} \\
\Sigma^{--}&=\begin{pmatrix} 0 &1\\ 1 &0 \end{pmatrix} \quad
Q_0=\begin{pmatrix} 1 &0\\ 0 &0 \end{pmatrix} \quad R_1=0
\end{align}
Figure $1$ pictures the $\mathcal{L}^2-$norm of the state $(u,v)$ in open loop, using the control law presented in \cite{hu2015control} and then using the control law \eqref{eq_U} presented in this paper. While the system in open loop is unstable (the $\mathcal{L}^2-norm$ diverges) it converges in minimum time $t_F=\dfrac{1}{\lambda_1}+\dfrac{1}{\mu_1}=2$ when controller \eqref{eq_U} is applied as expected from Theorem $2$. The controller presented in \cite{hu2015control} converges in a larger time which is equal (as mentioned in \cite{hu2015control}) to $\dfrac{1}{\lambda_1}+\dfrac{1}{\mu_1}+\dfrac{1}{\mu_2}=2.5$.
\begin{figure}[htb]
		\includegraphics[width=\columnwidth]{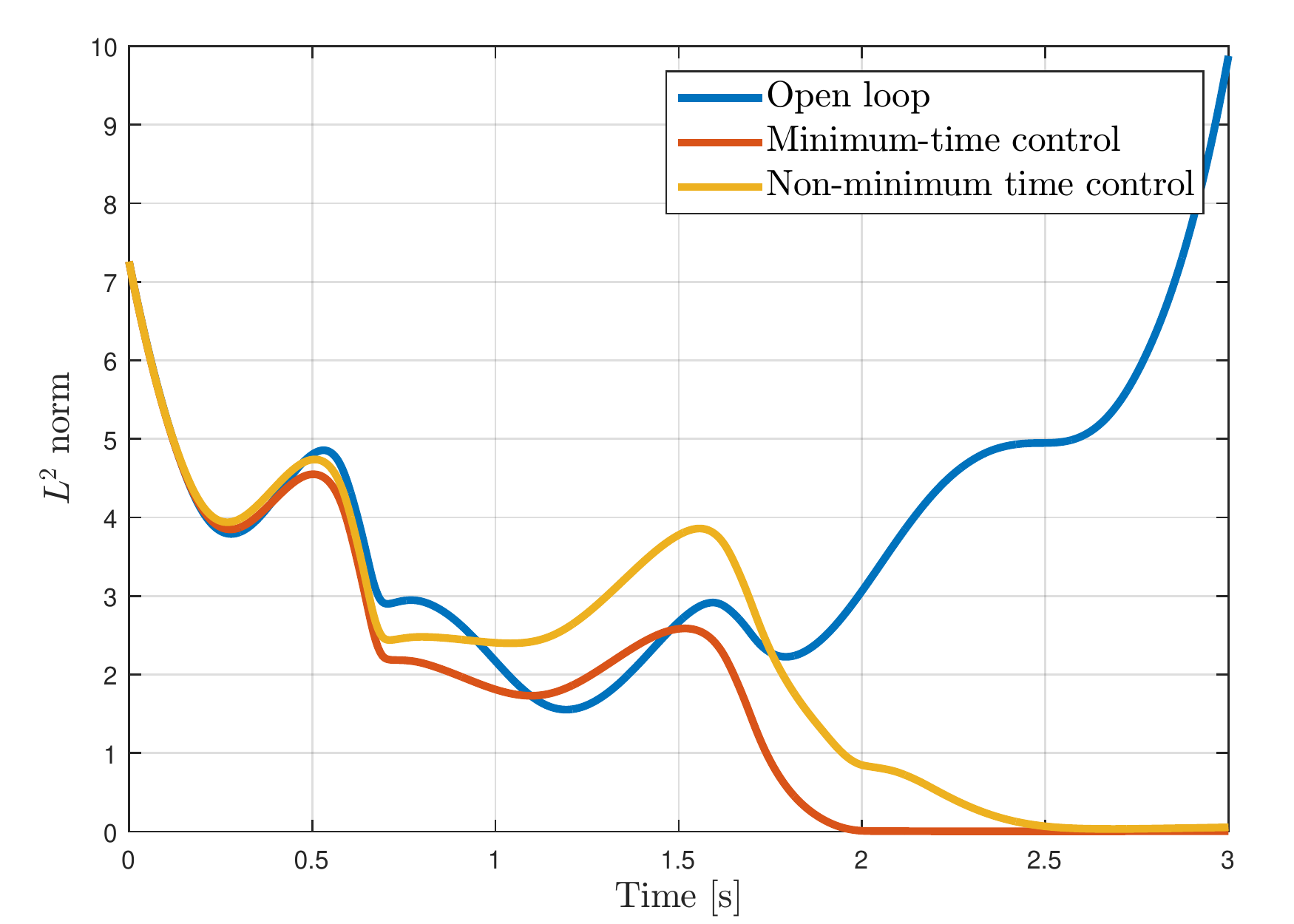}
	\caption{Time evolution of the L2-norm in open loop and using two different controlers}
	\label{fig:Simulation}
\end{figure}
\section{Concluding remarks}

Using the backstepping approach we have presented a stabilizating boundary feedback law for a general class of linear first-order system. Moreover, contrary to \cite{hu2015control}, the zero-equilibrium of the system is reached in minimum time $t_F$. 

The presented design raises several important questions that will be the topic of future investigation. In \cite{hu2015control}, the proposed control law does not yield minimum time convergence, but features several degrees of freedom that may be useable to handle transients. A comparison of the transient responses of both designs, as well as their comparative robustness, should be performed.

Besides, the presented result narrows the gap with the theoretical controllability results of \cite{li2010strong}. These results, although they do not provide explicit control law, ensure exact minimum-time controllability with less control inputs than what is currently achievable using backstepping. More generally, this raises the question of the links between stabilizability and stabilizability by backstepping.

\bibliographystyle{amsplain}
\bibliography{Biblio}

\end{document}